\documentclass[12pt, reqno]{amsart}
\usepackage{amssymb,dsfont}
\usepackage{eucal}
\usepackage{amsmath}
\usepackage{amscd}
\usepackage[dvips]{color}
\usepackage{multicol}
\usepackage[all]{xy}           %xypic macro for latex2.09
\usepackage{graphicx}
\usepackage{color}
\usepackage{colordvi}
\usepackage{xspace}
\usepackage{txfonts}
\usepackage{lscape}
\usepackage{tikz} % A REMETTRE

\numberwithin{equation}{section}
\usepackage[shortlabels]{enumitem}
\usepackage{ifpdf}
\ifpdf
 \usepackage[colorlinks,final, backref=page, hyperindex]{hyperref}
\else
  \usepackage[colorlinks,final,backref=page,hyperindex,hypertex]{hyperref}
\fi

\usepackage{tikz}
\usepackage[active]{srcltx}

\newtheorem{theorem}{Theorem}[section]
\newtheorem{lemma}[theorem]{Lemma}%[section]
\newtheorem{proposition}[theorem]{Proposition}%[section]
\newtheorem{corollary}[theorem]{Corollary}%[section]
\theoremstyle{definition}
\newtheorem{definition}[theorem]{Definition}
\newtheorem{remark}[theorem]{Remark}%[section]
\newtheorem{case}{Case}%[section]
\newtheorem{claim}{Claim}
\numberwithin{equation}{section}

\def\sl{\mathfrak{sl}}

\def\al{\alpha}
\def\be{\beta}

\def\gg{\mathfrak{g}}

\def \<{\langle}
\def \>{\rangle}

\def\GG{\mathcal{G}}

\def\VV{\mathcal{V}}

\def\WW{\mathcal{W}}

\newcommand{\C}{\mathbb {C}}

\newcommand{\Z}{\mathbb{Z}}

\allowdisplaybreaks

\begin{document}
\title[Anti-pre-Lie algebra]{Graded anti-pre-Lie algebraic structures on Witt and Virasoro algebras}

 \author{Chengming Bai}
  \address{C. Bai: Chern Institute of Mathematics and LPMC, Nankai University, Tianjin 300071, P. R. China}
  \email{baicm@nankai.edu.cn}

  \author{Dongfang Gao}
  \address{D. Gao: Chern Institute of Mathematics and LPMC, Nankai University, Tianjin 300071, P. R. China, and
 Institut Camille Jordan, Universit\'{e} Claude Bernard Lyon 1, Lyon, 69622, France}
  \email{gao@math.univ-lyon1.fr}


\maketitle

\begin{abstract}

We give the graded anti-pre-Lie algebraic structures on
the Witt algebra $\WW$ by the classification of certain
indecomposable weight representations of $\WW$. Their
classification in the sense of isomorphism is also given.
Furthermore, there does not exist a graded anti-pre-Lie
algebraic structure on the Virasoro algebra $\VV$ satisfying some
natural conditions.
\end{abstract}

\raggedbottom
\section{Introduction}

The notion of anti-pre-Lie algebras was introduced in \cite{LB}
as the underlying algebraic structures of nondegenerate
commutative 2-cocycles \cite{Dzh} which are the symmetric
version of symplectic forms on Lie algebras. They are regarded as
the ``anti-structures" of pre-Lie algebras in the sense that
anti-pre-Lie algebras are characterized as Lie-admissible algebras
whose negative left multiplication operators make representations
of the commutator Lie algebras, whereas pre-Lie algebras are
characterized as Lie-admissible algebras whose left
multiplication operators make representations of the commutator
Lie algebras. Note that pre-Lie algebras arose from the study of
deformations of associative algebras \cite{Ger}, affine
manifolds and affine structures on Lie groups \cite{Kos} and
convex homogeneous cones \cite{Vin}, and appeared in many fields
of mathematics and mathematical physics \cite{Bai, Bur} and the
references therein.

It is natural to consider whether there are
(anti-)pre-Lie algebraic structures on a fixed Lie algebra such
that the Lie algebra is the commutator of these (anti-)pre-Lie
algebras. Furthermore, if the answer is positive, the
classification of these (anti-)pre-Lie algebras should
be considered, too. In the case of pre-Lie algebras, there has
been some important results. For example, there does not exist a pre-Lie algebraic structure on any finite-dimensional
semisimple Lie algebra over an algebraically closed field of
characteristic zero \cite{Med}, whereas the classifications of graded pre-Lie algebraic structures on the Witt and
Virasoro algebras were given in \cite{KCB}. For the case of
anti-pre-Lie algebras, it seems that this problem is quite
complicated and hence there are few results. Note that the authors gave an example of
anti-pre-Lie algebraic structures on $\sl_2(\C)$ \cite{LB}. 
So it is necessary to consider the existence and the further
classification of anti-pre-Lie algebras on some concrete Lie algebras, as a guide for a further development.

In this paper, we study the classification of graded anti-pre-Lie algebraic structures on the Witt and Virasoro algebras. The Witt algebra is the complexification of the Lie algebra of real vector fields on the circle $S^1$  \cite{C} and the Virasoro algebra is the central extension of the Witt algebra  \cite{Bl, BT, GF, Vir}. Both them are important infinite-dimensional Lie algebras which play important roles in mathematics and mathematical physics such as quantum physics \cite{GO}, conformal field theory \cite{DMS} and vertex operator algebras \cite{DMZ, FZ}. The Witt and Virasoro algebras are under active investigations \cite{AKL} and they give the foundation for the construction of many interesting algebraic structures. For example, transposed Poisson algebraic structures \cite{BBGW} as an example of new algebraic structures with two or more multiplications with very active study in recent years \cite{K, Kh} and the references therein, were constructed on the Witt algebra \cite{FKL}. Moreover, the Witt algebra is a simple Lie algebra and hence any anti-pre-Lie algebraic structure on the Witt algebra is simple in the sense that there is not an ideal besides zero and itself, too. We would like to emphasize that our study bases on the well-established classification of certain indecomposable weight representations of the Witt algebra.

 The paper is organized as follows. In Section 2, we recall some notions and basic results on anti-pre-Lie algebras and the Witt
 algebra. In particular,  the classification of certain indecomposable weight representations of the Witt algebra is given. In Section 3, we show that a graded anti-pre-Lie algebraic structure on
the Witt algebra provides an indecomposable weight representation
of the Witt algebra by the negative left-multiplication operators
and the nonzero weight spaces of these indecomposable
weight representations are all one-dimensional. Thus all 
graded anti-pre-Lie algebraic structures on the Witt algebra are
obtained by the aforementioned classification of the
indecomposable weight representations of the Witt algebra.
  Moreover, their classification is also given.
In Section 4, we prove that there does not exist a 
graded anti-pre-Lie algebraic structure  on the Virasoro algebra
satisfying certain natural conditions.

Throughout this paper, we denote by $\Z, \Z^*,\Z_+, \C$ and $\C^*$ the
set of integers, nonzero integers, positive integers, complex numbers and nonzero complex numbers respectively.
All vector spaces and algebras are over $\C$, unless otherwise stated.

\section{Preliminaries on anti-pre-Lie algebras and the Witt algebra}
Some notions and basic results on anti-pre-Lie algebras and the
Witt algebra are given, including some results on the
indecomposable weight representations of the Witt algebra.

\begin{definition}\textup{(\cite{LB})}
An {\bf anti-pre-Lie algebra} is a vector space $A$ with a binary
operation $\circ$ satisfying
\begin{align}
&x\circ (y\circ z)-y\circ (x\circ z)=[y,x]\circ z,\label{anti-1}\\
&[[x,y], z]+[[y,z], x]+[[z,x], y]=0,\label{anti-2}
\end{align}
where $[x,y]=x\circ y-y\circ x$ for any $x,y,z\in A$.
\end{definition}

\begin{lemma}\textup{(\cite{LB})}\label{property-1}
Let $(A, \circ)$ be an anti-pre-Lie algebra.
For any $x\in A$, denote by $L_x$ the left-multiplication operator, that is, $L_x(y)=x\circ y$ for any $y\in A$.
Then the following results hold.
\begin{enumerate}%[$(1)$]
\item The commutator
$$[x,y]=x\circ y-y\circ x, \ \ \ \forall x,y\in A,$$
defines a Lie algebra, denoted by $\GG(A)$, which is called the
{\bf sub-adjacent Lie algebra} of $(A,\circ)$. Furthermore,
$(A,\circ)$ is called an {\bf anti-pre-Lie algebraic
structure} on the Lie algebra  $\GG(A)$. \item Let $\rho:
\GG(A)\rightarrow \frak g\frak l(A)$ be a linear map defined by
$\rho(x)=-L_x$, which is the negative left-multiplication
operator, for any $x\in\GG(A)$. Then
$\rho$ defines a representation of the Lie algebra $\GG(A)$.
\end{enumerate}
\end{lemma}

\begin{definition}\textup{(\cite{LB})}
An {\bf admissible Novikov algebra} is a vector space $A$ with a
binary operation $\circ$ satisfying Eq.~(\ref{anti-1}) and the
following equation \begin{equation}(x\circ y)\circ z-(x\circ
z)\circ y=2x\circ (y\circ z-z\circ y), \ \forall x,y,z\in A.
\end{equation}
\end{definition}

Admissible Novikov algebras are anti-pre-Lie algebras. Moreover,
there is a correspondence between admissible Novikov algebras and
Novikov algebras as a subclass of pre-Lie algebras introduced in
connection with Hamiltonian operators in the formal variational
calculus \cite{GD1,GD2} and Poisson brackets of hydrodynamic
type \cite{BN} in terms of $q$-algebras. That is,  the
$2$-algebra of a Novikov algebra is an admissible Novikov algebra
and the $(-2)$-algebra of an admissible Novikov algebra is a
Novikov algebra \cite{LB}.

Next we recall some results on the Witt algebra and its
indecomposable representations.

\begin{definition}\textup{(\cite{C})}
    The {\bf Witt algebra} denoted by $\WW$ is an infinite-dimensional Lie algebra with a basis $\{W_m~|~m\in\Z\}$ satisfying the following commutation relations
    $$[W_m, W_n]=(n-m)W_{m+n},\ \ \ \ \forall m,n\in\Z.$$
Moreover, the universal central extension $\VV$ of the Witt
algebra $\WW$, called {\bf Virasoro algebra},  is an
infinite-dimensional Lie algebra with a basis $\{W_m, {\bf
c}~|~m\in\Z\}$ satisfying the following commutation relations
    $$[W_m, W_n]=(n-m)W_{m+n}+\delta_{m+n,0}\frac{m^3-m}{12}{\bf c},\ \ \ \ [W_m, {\bf c}]=0,\ \ \ \ \forall m,n\in\Z.$$
\end{definition}

\begin{definition}
Let $\gg$ be a Lie algebra and $V$ be a representation of $\gg$. Then $V$ is called an {\bf indecomposable representation} of $\gg$ if $V$ can not be decomposed into  a direct sum of two proper subrepresentations.
\end{definition}

\begin{definition}
Let $V$ be a representation of $\WW$. Then $V$ is called a {\bf weight representation} of $\WW$ if $V=\oplus_{\lambda\in\C}V_\lambda$ as vector spaces, where $V_\lambda=\{v\in V~|~W_0v=\lambda v\}$ is called a {\bf weight space}.
\end{definition}

Now suppose that $V=\oplus_{i\in\Z}\C v_i$ is an
infinite-dimensional vector space. Thanks to \cite{KS} the
following statements hold.
\begin{enumerate}
\item For any $\al\in\C$, $V$ is a representation of $\WW$ with the following actions
$$W_mv_i=(m+i)v_{m+i}, \ \ \forall m\in\Z, i\in\Z^* \ \text{\ and \ } W_m v_0=m(\al+m)v_m, \ \ \forall m\in\Z.$$
Denote by $V_\al$ this representation.
\item For any $\be\in\C$, $V$ is a representation of $\WW$ with the following actions
$$W_mv_i=iv_{m+i}, \ \ \forall m,i\in\Z \text{\ with \ }m+i\ne 0 \ \text{\ and \ } W_i v_{-i}=-i(\be+i)v_0, \ \ \forall i\in\Z.$$
Denote by $V^\be$ this representation.
\item Let $\al,\be\in\C$ with $0\leq \mathrm{Re} \al<1$, where $\mathrm{Re} \al$ denotes the real part of $\al$. Then $V$ is a representation of $\WW$ with the following actions
$$W_mv_i=(\al+i+m\be)v_{m+i}, \ \ \ \forall m,i\in\Z.$$
Denote by $V_{\al,\be}$ this representation.
\end{enumerate}

\begin{theorem}\textup{(\cite{KS})}\label{indecomposable-W}
Let $\lambda\in\C$ and $V$ be an indecomposable weight
representation of $\WW$. Suppose that $V=\oplus_{i\in\Z}V_i$,
where $V_i=\{v\in V~|~W_0v=(\lambda+i)v\}$ and $\dim V_i=1$ for
any $i\in\Z$. Then $V$ is isomorphic to one of $V_\al, V^\be,
V_{\al,\be}$ as representations of $\WW$.
\end{theorem}

\section{Graded anti-pre-Lie algebraic structures on the Witt algebra}
We investigate the graded anti-pre-Lie algebraic
structures on the Witt algebra $\WW$ due to the aforementioned
classification of the indecomposable weight representations of
$\WW$ given in Theorem~\ref{indecomposable-W}. Their
classification in the sense of isomorphism is also given.

In this section, we study the anti-pre-Lie algebraic
structures on $\WW$ satisfying \begin{equation}\label{graded-eq}
W_m\circ
W_n=\varphi(m,n)W_{m+n},\ \ \ \ \forall m,n\in\Z,\end{equation}
where $\varphi:\Z\times \Z\rightarrow \C$ is a complex-valued
function.
Denote this
anti-pre-Lie algebra by $(S,\circ)$.

\begin{lemma} \label{lem:3.1} $(S, \circ)$ is an 
anti-pre-Lie algebraic structure on $\WW$ if and only if
$\varphi(m,n)$ satisfies the following two equations:
\begin{align}
&\varphi(m,n)-\varphi(n,m)=n-m,\label{linear-1}\\
&(n-m)\varphi(m+n,l)=\varphi(m,l)\varphi(n,m+l)-\varphi(n,l)\varphi(m,n+l),\
\ \ \forall m,n,l\in\Z.\label{linear-2}
\end{align}
\end{lemma}

\begin{proof} By definition, $(S, \circ)$ is an 
anti-pre-Lie algebraic structure on $\WW$ if and only if the
following equations hold:
\begin{align}
&W_m\circ W_n-W_n\circ W_m=(n-m)W_{m+n}, \label{compatible-1}\\
&W_m\circ(W_n\circ W_l)-W_n\circ(W_m\circ W_l)=(W_n\circ W_m)\circ W_l-(W_m\circ W_n)\circ W_l,\ \ \forall m,n,l\in\Z.  \label{compatible-2}
\end{align}
By Eq.~\eqref{graded-eq}, it is clear that Eqs.~\eqref{compatible-1} and \eqref{compatible-2} hold
if and only if Eqs.~\eqref{linear-1} and ~\eqref{linear-2} hold
respectively.
\end{proof}

From Eqs.~\eqref{linear-1} and \eqref{linear-2} we have the
following simple observations.
\begin{enumerate}
    \item Taking $l=0$ in Eq.~\eqref{linear-2}, we have
    \begin{align}\label{l=0}
    (n-m)\varphi(m+n,0)=\varphi(m,0)\varphi(n,m)-\varphi(n,0)\varphi(m,n), \ \ \ \forall m,n\in\Z.
    \end{align}
    \item Taking $l=1$ in Eq.~\eqref{linear-2},  we have
    \begin{align}\label{l=1}
    (n-m)\varphi(m+n,1)=\varphi(m,1)\varphi(n,m+1)-\varphi(n,1)\varphi(m,n+1),\ \ \ \forall m,n\in\Z.
    \end{align}
    \item Taking $l=2$ in Eq.~\eqref{linear-2},  we have
    \begin{align}\label{l=2}
    (n-m)\varphi(m+n,2)=\varphi(m,2)\varphi(n,m+2)-\varphi(n,2)\varphi(m,n+2), \ \ \ \forall m,n\in\Z.
    \end{align}
    \item Taking $m=0$ in Eq.~\eqref{linear-2}, we have
    \begin{align}\label{m=0}
    (\varphi(0,l)-\varphi(0,n+l)-n)\varphi(n,l)=0, \ \ \ \forall n,l\in\Z.
    \end{align}
    \item Taking $m=l=0$ in Eq.~\eqref{linear-2}, we have
    \begin{align}\label{m=l=0}
    (\varphi(0,0)-\varphi(0,n)-n)\varphi(n,0)=0, \ \ \ \forall n\in\Z.
    \end{align}
\end{enumerate}
By Eq.~\eqref{m=l=0}, we know that
\begin{equation}\label{eq:11}\varphi(n,0)=0 \text{\ \ \ or\ \ \  }
\varphi(0,n)+n=\varphi(0,0),\ \ \ \forall n\in\Z.\end{equation}
By
Eq.~\eqref{linear-1}, the following equation holds:
\begin{equation}
\varphi(n,0)=0 \text{\ \ \ or\ \ \ }
\varphi(n,0)+2n=\varphi(0,0),\ \ \ \forall n\in\Z.
\end{equation}

Set
$$\Gamma_1=\{m\in\Z~|~\varphi(m,0)=0\},\ \ \ \Gamma_2=\{m\in\Z~|~\varphi(m,0)+2m=\varphi(0,0)\}.$$
Obviously, by Eq.~\eqref{l=0}, we show that for any
$m,n\in\Gamma_1$ and $m\ne n$, $m+n\in\Gamma_1$.

\begin{lemma}\label{notin2Z}
If $\varphi(0,0)\notin 2\Z$, then $\Gamma_2=\Z$.
\end{lemma}
\begin{proof}
First, we have the following conclusions.
\begin{equation}\Gamma_1\cap\Gamma_2=\emptyset,
\;\;\Gamma_1\cup\Gamma_2=\Z,\;\; 0\in\Gamma_2.\label{cap}
\end{equation}

\begin{claim}\label{m+ninGamma2}
For any $m,n\in\Gamma_2$ and $m\ne n$, $m+n\in\Gamma_2$.
\end{claim}
By Eqs.~\eqref{linear-1} and \eqref{l=0}, we have
\begin{align*}
(n-m)\varphi(m+n,0)
&=\varphi(0,0)\varphi(n,m)-2m\varphi(n,m)-\varphi(0,0)\varphi(m,n)+2n\varphi(m,n)\\
&=\varphi(0,0)(m-n)-2m(m-n)+2(n-m)\varphi(m,n),
\end{align*}
which implies
\begin{align*}
\varphi(m+n,0)=2\varphi(m,n)-\varphi(0,0)+2m.
\end{align*}

\begin{enumerate}
\item[(i)] If $\varphi(m,n)=0$, then
$$\varphi(m+n,0)=-\varphi(0,0)+2m.$$ Since $\varphi(0,0)\notin
2\Z$,  $\varphi(m+n,0)\ne 0$. Hence
 $m+n\notin \Gamma_1$. So $m+n\in \Gamma_2$.

\item[(ii)] If $\varphi(m,n)\ne 0$, then by Eq.~\eqref{m=0}, we
have
    $$\varphi(0,n)-\varphi(0,m+n)-m=0.
    $$
By Eq.~\eqref{linear-1} and the assumption that $n\in\Gamma_2$, we
have
$$\varphi(m+n,0)+2m+2n-\varphi(0,0)=0.$$
     So $m+n\in\Gamma_2$.
\end{enumerate}
Therefore Claim \ref{m+ninGamma2} is proved.

\begin{claim}\label{1inGamma2}
$1\in\Gamma_2$.
\end{claim}
Assume that $1\notin \Gamma_2$. That is, $1\in \Gamma_1$.
 Then $\varphi(0,1)=1$. Taking $m=1, n=-1$ in Eq.~\eqref{l=1}, we get
$$-2=-2\varphi(0,1)=\varphi(1,1)\varphi(-1,2).$$ Hence $\varphi(1,1)\ne 0, \varphi(-1,2)\ne 0$.

Let $m=0, n=1$ in Eq.~\eqref{l=1}. Then we have
$$\varphi(1,1)=\varphi(0,1)\varphi(1,1)-\varphi(1,1)\varphi(0,2).$$
Hence $\varphi(0,2)=0$. So $\varphi(2,0)=-2$ and thus
$2\notin\Gamma_1$ and $2\notin \Gamma_2$, which is a
contradiction.

Therefore Claim \ref{1inGamma2} is proved.

\begin{claim} \label{2inGamma2}
    $2\in\Gamma_2$.
\end{claim}
Assume that $2\notin\Gamma_2$. That is, $2\in \Gamma_1$. Then
$\varphi(0,2)=2$. Taking $m=2, n=-2$ in Eq.~\eqref{l=2}, we have
$$-8=-4\varphi(0,2)=\varphi(2,2)\varphi(-2,4).$$
Hence $\varphi(2,2)\ne 0, \varphi(-2,4)\ne 0$.

Let $m=0, n=2$ in Eq.~\eqref{l=2}. Then we get
$$2\varphi(2,2)=\varphi(0,2)\varphi(2,2)-\varphi(2,2)\varphi(0,4).$$
Hence $\varphi(0,4)=0$. So $\varphi(4,0)=-4$ and thus
$4\notin\Gamma_1$ and $4\notin \Gamma_2$, which is a
contradiction.

Therefore Claim \ref{2inGamma2} is proved.

Similarly, we prove that $\{-1,-2\}\subseteq\Gamma_2$. Therefore
$\Gamma_2=\Z$.
\end{proof}

\begin{lemma}\label{2N}
    If $\varphi(0,0)=2N$, where $N\in\Z$, then $\Gamma_2=\Z$.
\end{lemma}
\begin{proof}
First, obviously, the following conclusions hold.
\begin{equation}
\Gamma_1\cap\Gamma_2=\{N\}, \Gamma_1\cup\Gamma_2=\Z,
0\in\Gamma_2.\label{cap'}\end{equation}

\begin{claim}\label{m+ninGamma2'}
    Let $m, n\in\Gamma_2$ satisfying  $m\ne n$ and $m\ne N$. Then $m+n\in\Gamma_2$.
\end{claim}
By Eq.~\eqref{l=0}, we have
\begin{align*}
\varphi(m+n,0)=2\varphi(m,n)-\varphi(0,0)+2m.
\end{align*}

\begin{enumerate} \item[(i)] If $\varphi(m,n)=0$, then
$\varphi(m+n,0)=-\varphi(0,0)+2m=-\varphi(m,0)$. Since $m\ne N$,
$\varphi(m+n,0)\ne 0$ and hence $m+n\notin \Gamma_1$. So $m+n\in
\Gamma_2$.

\item[(ii)] If $\varphi(m,n)\ne 0$, then by Eq.~\eqref{m=0}, we
get
    $$\varphi(0,n)-\varphi(0,m+n)-m=0.$$
  By Eq.~\eqref{linear-1} and the assumption that $n\in\Gamma_2$, we
have $$\varphi(m+n,0)+2m+2n-\varphi(0,0)=0.$$ So $m+n\in\Gamma_2$.
\end{enumerate}

   Therefore  Claim \ref{m+ninGamma2'} is proved.

Note that in Claim \ref{m+ninGamma2'}, the appearance of $m$ and
$n$ is symmetric and hence we show that for any $m, n\in\Gamma_2$
and $m\ne n$, $m+n\in\Gamma_2$. Moreover, by similar proofs as the
ones for Claim \ref{1inGamma2} and Claim \ref{2inGamma2}, we show
that $\{1,-1,2,-2\}\in\Gamma_2$. So $\Gamma_2=\Z$.
\end{proof}

Combining Lemmas~\ref{notin2Z} and~\ref{2N} together, we have the
following conclusion.

\begin{corollary}\label{Gamma2}
$\varphi(m,0)+2m=\varphi(0,0)$ for any $m\in\Z$.
\end{corollary}

\begin{lemma}\label{S} Suppose that $(S, \circ)$ is a 
graded anti-pre-Lie algebraic structure on $\WW$. Define a linear
map $\rho:\WW\rightarrow \frak g\frak l(S)$ by
\begin{align}\label{action}
 \rho(W_m) W_n =   -W_m\circ W_n=-\varphi(m,n)W_{m+n}, \ \ \ \forall m,n\in\Z.
    \end{align}
Then $(\rho, S)$ is a weight representation of $\WW$ in which
    each nonzero weight space is one-dimensional.
\end{lemma}
\begin{proof}
By Lemma \ref{property-1}, we show that $(\rho, S)$ is a
representation of $\WW$.  Note that by Eq.~\eqref{linear-1} and
Corollary \ref{Gamma2}, we have
\begin{align*}
\rho(W_0)W_n=-W_0\circ W_n=-\varphi(0,n)W_n=(n-\varphi(0,0))W_n,
\end{align*}
for any $n\in\Z$. Thus $(\rho, S)$ is a weight representation of
$\WW$ in which each nonzero weight space is one-dimensional.
\end{proof}

\begin{proposition}\label{indecomposable} With the assumptions and notations in
Lemma~\ref{S}, $(\rho, S)$ is an indecomposable
representation of $\WW$. Moreover, $(\rho, S)$ is isomorphic to
one of $V_{\al}, V^{\be}, V_{\al, \be}$.
\end{proposition}

\begin{proof}
    Assume that $(\rho, S)$ is decomposable. That is, there exist two nonzero proper subrepresentations  $S_1$ and $S_2$ of $S$ such that $S=S_1\oplus S_2$.
    It is known that any subrepresentation of a weight representation of $\WW$ is also a weight representation. So there exist two proper subsets $J_1, J_2$ of $\Z$ such that $$J_1\cap J_2=\emptyset, \ \ \ J_1\cup J_2=\Z, \ \ \ S_j=\oplus_{m\in J_j}\C W_m,\ \ \ j=1,2.$$
Without loss of generality, we assume that $W_0\in S_1$, that is,
$0\in J_1$.

Now we will prove $\Z_+\subseteq J_1$.

Assume that there exists a positive integer in $J_2$. Let $s$ be
the smallest positive integer in $J_2$. Then the following
conclusions hold.
    \begin{enumerate}
        \item\label{indecom-1} For any $m\in\Z$, $\varphi(m,0)=2s-2m, \varphi(0,m)=2s-m$. Since $W_0\in S_1$ and $W_s\in S_2$, we have $$\rho(W_s) W_0=-\varphi(s,0)W_s\in S_1\cap S_2.$$  So $\varphi(s,0)=0$ and thus $\varphi(0,0)=2s$.
        Hence
        $$\varphi(m,0)=\varphi(0,0)-2m=2s-2m, \ \ \ \varphi(0,m)=\varphi(m,0)+m=2s-m, \ \ \ \forall m\in\Z.$$
        \item\label{indecom-2} For any $m,n\in J_2$ and $m\ne n$, $m+n\in J_2$. In fact, by Eq.~\eqref{linear-1} we have $n-m=\varphi(m,n)-\varphi(n,m)$. Thus $\varphi(m,n)\ne 0$ or $\varphi(n,m)\ne 0$, which imply
        $$\rho(W_m) W_n=-\varphi(m,n)W_{m+n}\ne 0, \text{\ \ \ or\ \ \ } \rho(W_n)W_m=-\varphi(n,m)W_{m+n}\ne 0.$$
        So $m+n\in J_2$.
        \item\label{indecom-3} If $m\in J_2$, then $-m\in J_1$. Otherwise we deduce that $0=m+(-m)\in J_2$ which is a contradiction.
        \item\label{indecom-4} Let $m\in J_j$ $(j=1,2)$. For any $n\in\Z$, if $\varphi(n,m)\ne 0$, then $m+n\in J_j$ due to Eq.~\eqref{action}.
        \item\label{indecom-5} Let $m\in J_2, n\in J_1$. Then by Eq.~\eqref{action}, we show that $m+n\in J_1$ if and only if $\varphi(n,m)=0, \varphi(m,n)=n-m$, and $m+n\in J_2$ if and only if $\varphi(m,n)=0, \varphi(n,m)=m-n$.
        \item\label{indecom-6} $s+1\in J_2$. Taking $m=-n=-1, l=s$ in Eq.~\eqref{linear-2}, we
        have
        \begin{align}\label{s+1}
        2\varphi(0,s)=\varphi(-1, s)\varphi(1,s-1)-\varphi(1, s)\varphi(-1,s+1).
        \end{align}
        Since $\varphi(0,s)=s\ne 0$ and $s-1\notin J_2$, $\varphi(1,s)\ne 0$. Hence by Eq.~\eqref{action}, $s+1\in
        J_2$.
    \end{enumerate}
So we give the following interpretation.
\begin{enumerate}
  \item[(i)]  If $s=1$, then by Items \eqref{indecom-3} and \eqref{indecom-6}, we have $\{1,2\}\subseteq J_2, \{0,-1,-2\}\subseteq J_1$. Moreover, by Item \eqref{indecom-5}, Eq.~\eqref{s+1} becomes
    $$  2=-3\varphi(1, 1).$$
    So $\varphi(1,1)=-\frac{2}{3}$. Let $m=-n=2$ in Eq.~\eqref{l=1}. Then we deduce $\varphi(2,1)=-\frac{4}{5}$.
    Let $m=2,n=-1$ in Eq.~\eqref{l=1}. Then we have
    $$-3\varphi(1,1)=\varphi(2,1)\varphi(-1,3)-\varphi(-1,1)\varphi(2,0).$$
    Hence $2=-\frac{16}{5}$,  which is a contradiction.
\item[(ii)]    If $s>1$, then we have $W_{s-1}, W_1\in S_1$. Thus
    $$\rho(W_1)W_{s-1}\in S_1\cap S_2 \text{\ \ \ and \ \ \ }  \rho(W_{s-1}) W_1\in S_1\cap S_2.$$
    So $\varphi(1, s-1)=\varphi(s-1,1)=0$ and thus $s=2$. Now $\{1,0,-1,-2\}\subseteq J_1, 3\in J_2$. So $\varphi(1,2)=1$.
    Let $m=-n=2$ in Eq.~\eqref{l=2}. Then we get
    $$-4\varphi(0,2)=\varphi(2,2)\varphi(-2,4)-\varphi(-2,2)\varphi(2,0).$$
    Hence $\varphi(2,2)=-\frac{4}{3}$. So $4\in J_2$.
    Let $m=2, n=-1$ in Eq.~\eqref{l=2}. Then we have
    $$-3\varphi(1,2)=\varphi(2,2)\varphi(-1,4)-\varphi(-1,2)\varphi(2,1).$$
    Hence $3=\frac{20}{3}$, which is a contradiction.
\end{enumerate}

Therefore, all positive integers are in $J_1$. By a similar
interpretation we show that all negative integers are in $J_1$.
Thus $J_2=\emptyset$ which is a contradiction. So $(\rho, S)$ is
indecomposable. By Theorem~\ref{indecomposable-W},  $(\rho, S)$ is
isomorphic to one of $V_{\al}, V^{\be}, V_{\al, \be}$.
\end{proof}

\begin{lemma}\label{A-al} With the assumptions and notations in
Lemma~\ref{S}, $(\rho, S)$ is not isomorphic to $V_\al$ as
$\WW$-representations.
\end{lemma}
\begin{proof}
Assume that $f: S\rightarrow V_\al$ is a $\WW$-representation
isomorphism. Set $$f(W_0)=\sum_{i\in\Z}a_iv_i,$$ where $a_i\in\C$
and only finitely many of $a_i$ are nonzero. Then we have
\begin{align}\label{A-1}
\sum_{i\in\Z^*}ia_iv_i=W_0f(W_0)=f(\rho(W_0)W_0)=-\varphi(0,0)f(W_0)=-\varphi(0,0)\sum_{i\in\Z}a_iv_i.
\end{align}
\begin{case}
There exists $k\in\Z^*$ such that $a_k\ne 0$.
\end{case}
In this case,  $\varphi(0,0)=-k\ne 0$ and $f(W_0)=a_kv_k$. Hence
we have
$$-kf(W_{-k})=-\varphi(-k,0)f(W_{-k})=f(\rho(W_{-k}) W_0)=W_{-k}f(W_0)=W_{-k}(a_kv_k)=0.$$
Thus $f(W_{-k})=0$, which is a contradiction.
\begin{case}
$f(W_0)=a_0v_0$, where $a_0\ne 0$.
\end{case}
In this case,  by Eq.~\eqref{A-1}, we have $\varphi(0,0)=0$.  For
any $m\ne 0$, we have
\begin{align*}
2mf(W_m)=-\varphi(m,0)f(W_m)=f(\rho(W_m)
W_0)=W_mf(W_0)=W_m(a_0v_0)=m(\al+m)a_0v_m.
\end{align*}
Hence
\begin{align}\label{A-2}
f(W_m)=\frac{\al+m}{2}a_0v_m, \ \ \ \forall m\ne 0.
\end{align}
Suppose that $m, m+n\in\Z^*$ and $\al+m+n\ne 0$. By
Eq.~\eqref{A-2}, we have
\begin{align*}
-\varphi(n,m)\frac{\al+m+n}{2}a_0v_{m+n}=f(\rho(W_n) W_m)=W_nf(W_m)=\frac{(\al+m)(m+n)}{2}a_0v_{m+n},
\end{align*}
which implies $$\varphi(n,m)=-\frac{(\al+m)(m+n)}{\al+m+n}.$$
Similarly, we show that
$$\varphi(m,n)=-\frac{(\al+n)(m+n)}{\al+m+n}$$ for any $n,
m+n\in\Z^*$ with $\al+m+n\ne 0$. Then by Eq.~\eqref{linear-1}, we
get
$$m-n=\frac{m+n}{\al+m+n}(n-m), \ \ \ \ \forall m, n, m+n\in\Z^* \text{\ and\ } \al+m+n\ne 0,$$
which is a contradiction.

Combining the arguments of the above two cases, we show that
$(\rho, S)$ is not isomorphic to $V_\al$ as $\WW$-representations.
\end{proof}

\begin{lemma}\label{B-be} With the assumptions and notations in
Lemma~\ref{S}, $(\rho, S)$ is not isomorphic to $V^\be$ as
$\WW$-representations.
\end{lemma}
\begin{proof}
Assume that $f:S\rightarrow V^\be$ is a $\WW$-representation
isomorphism. Set $$f(W_0)=\sum_{i\in\Z}a_iv_i,$$ where $a_i\in\C$
and only finitely many of $a_i$ are nonzero. Then we get
\begin{align}\label{B-1}
-\varphi(0,0)\sum_{i\in\Z}a_iv_i=-\varphi(0,0)f(W_0)=f(\rho(W_0) W_0)=W_0f(W_0)=W_0(\sum_{i\in\Z}a_iv_i)=\sum_{i\in\Z^*}ia_iv_i.
\end{align}
\setcounter{case}{0}
\begin{case}
There exists $k\in\Z^*$ such $a_k\ne 0$.
\end{case}
In this case, $\varphi(0,0)=-k\ne 0$ and $f(W_0)=a_kv_k$. Let
$m\ne -k$. Then we have
\begin{align*}
(2m+k)f(W_m)=-\varphi(m,0)f(W_m)=f(\rho(W_m)
W_0)=W_mf(W_0)=ka_kv_{m+k}\ne 0.
\end{align*}
Hence $k$ is odd and
 $$f(W_m)=\frac{k}{2m+k}a_kv_{m+k}, \ \ \ \forall m\in\Z \text{\ and\ } m+k\ne 0.$$
Taking $m,n\in\Z$ such that $m+k\ne 0, m+n+k\ne 0$, we have
\begin{align*}
-\varphi(n,m)\frac{k}{2m+2n+k}a_kv_{m+n+k}=-\varphi(n,m)f(W_{m+n})=f(\rho(W_n) W_m) \\
=W_nf(W_m)=\frac{k}{2m+k}a_kW_n(v_{m+k})=\frac{k(m+k)}{2m+k}a_kv_{m+n+k}.
\end{align*}
Hence we have
$$\varphi(n,m)=-\frac{m+k}{2m+k}(2m+2n+k),\ \ \ \ \forall m,n\in\Z \text{\ with \ } m+k\ne 0, m+n+k\ne 0.$$
Similarly, we have
$$\varphi(m,n)=-\frac{n+k}{2n+k}(2m+2n+k),\ \ \ \ \forall m,n\in\Z \text{\ with \ } n+k\ne 0, m+n+k\ne 0.$$
Thus
$$m-n=\varphi(n,m)-\varphi(m,n)=\frac{k(m-n)(2m+2n+k)}{4mn+2mk+2nk+k^2}.$$
Hence $4mn=0$ for any $m,n\in\Z$ with $m+k\ne 0, n+k\ne 0,
m+n+k\ne 0$, which is a contradiction.
\begin{case}
$f(W_0)=a_0v_0$, where $a_0\ne 0$.
\end{case}
In this case,  by Eq.~\eqref{B-1}, we have $\varphi(0,0)=0$. Let
$m\ne 0$. Then we have
$$0=W_mf(W_0)=f(\rho(W_m) W_0)=-\varphi(m,0)f(W_m)=2mf(W_m).$$
So $f(W_m)=0$, which is a contradiction.

Combining the arguments of the above two cases, we show that
$(\rho, S)$ is not isomorphic to $V^\be$ as $\WW$-representations.
\end{proof}

\begin{lemma}\label{be-ne-2} With the assumptions and notations in
Lemma~\ref{S}, if $\be\ne 2$, then $(\rho, S)$ is not isomorphic
to $V_{\al, \be}$ as $\WW$-representations.
\end{lemma}
\begin{proof}
Assume that $f: S\rightarrow V_{\al,\be}$ is a
$\WW$-representation isomorphism. Set
$$f(W_0)=\sum_{i\in\Z}a_iv_i,$$ where $a_i\in\C$ and only finitely
many of $a_i$ are nonzero. Then we get
\begin{align*}
-\varphi(0,0)\sum_{i\in\Z}a_iv_i=-\varphi(0,0)f(W_0)=f(\rho(W_0) W_0)=W_0f(W_0)=W_0(\sum_{i\in\Z}a_iv_i)=\sum_{i\in\Z}(\al+i)a_iv_i.
\end{align*}
So there exists $k\in\Z$ such that $\varphi(0,0)=-\al-k$ and
$f(W_0)=a_kv_k$, where $a_k\ne 0$. Let $m\ne 0$. Then we get
\begin{align*}
(2m+\al+k)f(W_m)=-\varphi(m,0)f(W_m)&=f(\rho(W_m) W_0)  \\
&=W_mf(W_0)=W_m(a_kv_k)=(\al+k+m\be)a_kv_{m+k}.
\end{align*}
Since $\be \ne 2$, it is straightforward to show that $2m+\al+k\ne
0$. Thus
\begin{align}\label{f(W_m)}
f(W_m)=\frac{\al+k+m\be}{\al+k+2m}a_kv_{m+k}, \ \ \ \ \forall m\ne 0.
\end{align}
So for any $m\ne 0$, we have
\begin{align}
-\varphi(m,-m)a_kv_k=-\varphi(m,-m)f(W_0)&=f(\rho(W_m) W_{-m})  \notag \\
&=W_mf(W_{-m})=\frac{(\al+k-\be m)(\al+k-m+m\be)}{\al+k-2m}a_kv_k.
\end{align}
Hence we have
$$\varphi(m,-m)=-\frac{(\al+k-\be m)(\al+k-m+m\be)}{\al+k-2m}, \ \ \ \forall m\ne 0.$$
Similarly, we get
$$\varphi(-m,m)=-\frac{(\al+k+\be m)(\al+k+m-m\be)}{\al+k+2m}, \ \ \ \forall m\ne 0.$$
By Eq.~\eqref{linear-1} we obtain
$$-2m=\varphi(m,-m)-\varphi(-m,m)=-\frac{4m^3\be(1-\be)+2m(\al+k)^2}{(\al+k+2m)(\al+k-2m)}, \ \ \ \forall m\ne0.$$
Hence $\be^2-\be-2=0$. Since $\be\ne 2$, we have $\be =-1$.

By Eq.~\eqref{f(W_m)}, we have
\begin{align}\label{f(W_m)'}
f(W_m)=\frac{\al+k-m}{\al+k+2m}a_kv_{m+k}, \ \ \ \ \forall m\ne 0.
\end{align}
So $\al+k\notin\Z^*$. (Otherwise, there exists $m\in\Z^*$ such
that $f(W_m)=0$ which is a contradiction). Taking $m\ne 0, m+n\ne
0$, we have
\begin{align*}
-\varphi(n,m)\frac{\al+k-m-n}{\al+k+2m+2n}a_kv_{m+n+k}&=-\varphi(n,m)f(W_{m+n})
=f(\rho(W_n) W_m)\\
&=W_nf(W_m)=\frac{\al+k-m}{\al+k+2m}a_kW_n(v_{m+k})  \\
&=\frac{(\al+k-m)(\al+k+m-n)}{\al+k+2m}a_kv_{m+n+k}.
\end{align*}
Hence we have
$$\varphi(n,m)=-\frac{(\al+k-m)(\al+k+2m+2n)}{(\al+k+2m)(\al+k-m-n)}(\al+k+m-n),  \ \ \ \ \forall m\ne 0, m+n\ne 0.$$
Similarly, we have
$$\varphi(m,n)=-\frac{(\al+k-n)(\al+k+2m+2n)}{(\al+k+2n)(\al+k-m-n)}(\al+k+n-m),  \ \ \ \ \forall n\ne 0, m+n\ne 0.$$
Therefore, for any $m\ne 0, n\ne 0, m+n\ne 0,m-n\ne 0$, by
Eq.~\eqref{linear-1}, we have
\begin{align}\label{m-n}
m-n&=\varphi(n,m)-\varphi(m,n)  \notag \\
&=\frac{\al+k+2m+2n}{\al+k-m-n}{\big ( }\frac{(\al+k-n)(\al+k+n-m)}{\al+k+2n}-\frac{(\al+k-m)(\al+k+m-n)}{\al+k+2m}
{\big )}.
\end{align}
Therefore $mn(m+n)=0$, which is a contradiction. So $(\rho, S)$ is
not isomorphic to $V_{\al, \be}$ as $\WW$-representations when
$\be\ne 2$.
\end{proof}

\begin{theorem}\label{main}
Suppose that $(S, \circ)$ is a graded anti-pre-Lie
algebraic structure on the Witt algebra $\WW$. Then there exists
$\gamma\in\C$ such that
\begin{align}\label{main-1}
W_n\circ W_m=-(\gamma+m+2n)W_{m+n}, \ \ \ \forall m,n\in\Z.
\end{align}
Moreover, for any $\gamma\in \mathbb C$, Eq.~\eqref{main-1}
defines a graded anti-pre-Lie algebraic structure on
$\WW$, which is denoted by $(S,\circ_\gamma)$.
\end{theorem}
\begin{proof}
On the one hand, it is straightforward  to show that
Eq.~\eqref{main-1} defines a graded anti-pre-Lie
algebraic structure on $\WW$. On the other hand, suppose that
$(S,\circ)$ is a graded anti-pre-Lie algebraic
structure on $\WW$. Then by Proposition \ref{indecomposable} and
Lemmas ~\ref{A-al}, \ref{B-be}, \ref{be-ne-2}, we show that
$(\rho, S)$ is isomorphic to $V_{\al, 2}$ as
$\WW$-representations, where $\al\in\C$ and $0\leq Re \al< 1$. Let
$f:S\rightarrow V_{\al, 2}$ be the $\WW$-representation
isomorphism. Set $$f(W_0)=\sum_{i\in\Z}a_iv_i,$$ where $a_i\in\C$
and only finitely many of $a_i$ are nonzero. Then we have
\begin{align*}
-\varphi(0,0)\sum_{i\in\Z}a_iv_i=-\varphi(0,0)f(W_0)=f(\rho(W_0) W_0)=W_0f(W_0)=W_0(\sum_{i\in\Z}a_iv_i)=\sum_{i\in\Z}(\al+i)a_iv_i.
\end{align*}
So there exists $k\in\Z$ such that $\varphi(0,0)=-\al-k$ and $f(W_0)=a_kv_k$, where $a_k\ne 0$.
For any $m\ne 0$, we get
\begin{align}\label{main-3}
(\al+k+2m)f(W_m)=-\varphi(m,0)f(W_m)&=f(\rho(W_m) W_0)  \notag \\
&=W_mf(W_0)=W_m(a_kv_k)=(\al+k+2m)a_kv_{m+k}.
\end{align}
\setcounter{case}{0}
\begin{case}\label{k-odd}
$\al\ne 0$ or $k$ is an odd number.
\end{case}
In this case, $\al+k+2m\ne 0$ for any $m\ne 0$. So we have
$$f(W_m)=a_kv_{m+k}, \ \ \ \forall m\in\Z.$$ Therefore, for any
$m,n\in\Z$, we have
\begin{align}
-\varphi(n,m)a_kv_{m+n+k}&=-\varphi(n,m)f(W_{m+n})=f(\rho(W_n) W_m)   \notag  \\
&=W_nf(W_m)=W_n(a_kv_{m+k})=(\al+k+m+2n)a_kv_{m+n+k}.
\end{align}
Thus $\varphi(n,m)=-(\al+k+m+2n)$. So $$W_n\circ
W_m=-(\al+k+m+2n)W_{m+n}, \ \ \ \forall m,n\in\Z.$$
\begin{case}
$\al=0, k=0$.
\end{case}
In this case,  by Eq.~\eqref{main-3}, we also have
$$f(W_m)=a_kv_{m+k}, \ \ \ \forall m\in\Z.$$
By a similar interpretation,  we have
$$W_n \circ W_m=-(\al+k+m+2n)W_{m+n},\ \ \ \forall m,n\in\Z.$$
\begin{case}
$\al=0, k\ne 0$ and $k$ is an even number.
\end{case}
In this case, by Eq.~\eqref{main-3}, we get
\begin{align}\label{main-7}
f(W_m)=a_kv_{m+k}, \ \ \ \forall m\ne -\frac{k}{2}.
\end{align}
Set $$f(W_{-\frac{k}{2}})=\sum_{i\in\Z}b_iv_i,$$ where $b_i\in\C$
and only finitely many of $b_i$ are nonzero. Then we have
\begin{align}
\frac{k}{2}\sum_{i\in\Z}b_iv_i=-\varphi(0,-\frac{k}{2})f(W_{-\frac{k}{2}})=f(\rho(W_0)
W_{-\frac{k}{2}})=W_0f(W_{-\frac{k}{2}})=\sum_{i\in\Z}ib_iv_i.
\end{align}
Hence there exists $c\in\C^*$ such that
\begin{align}\label{main-6}
f(W_{-\frac{k}{2}})=cv_{\frac{k}{2}}.
\end{align}
Then
\begin{align}
&-\varphi(\frac{k}{2},-\frac{k}{2})a_kv_k=-\varphi(\frac{k}{2},-\frac{k}{2})f(W_0)=f(\rho(W_{\frac{k}{2}}) W_{-\frac{k}{2}})
=W_{\frac{k}{2}}f(W_{-\frac{k}{2}})=\frac{3}{2}kcv_k,   \label{main-4}\\
&-\varphi(-\frac{k}{2},\frac{k}{2})a_kv_k=-\varphi(-\frac{k}{2},\frac{k}{2})f(W_0)=f(\rho(W_{-\frac{k}{2}}) W_{\frac{k}{2}})
=W_{-\frac{k}{2}}f(W_{\frac{k}{2}})=\frac{1}{2}ka_kv_k.  \label{main-5}
\end{align}
By Eqs.~\eqref{main-4} and \eqref{main-5}, we get
$$\varphi(\frac{k}{2},-\frac{k}{2})=-\frac{3kc}{2a_k}, \ \ \ \  \ \ \varphi(-\frac{k}{2},\frac{k}{2})=-\frac{k}{2}.$$
By Eq.~\eqref{linear-1}, we have $c=a_k$. Then by
Eqs.~\eqref{main-7} and \eqref{main-6}, we have
$$f(W_m)=a_kv_{m+k}, \ \ \ \forall m\in\Z.$$
By a similar interpretation for Case \ref{k-odd}, we get
$$W_n\circ W_m=-(\al+k+m+2n)W_{m+n},\ \ \ \forall m,n\in\Z.$$

Finally, replacing $\al+k$ by $\gamma$, we obtain
$$W_n\circ W_m=-(\gamma+m+2n)W_{m+n},\ \ \ \forall m,n\in\Z.$$
This completes the proof.
\end{proof}

\begin{remark}
\begin{enumerate}
\item For any $\gamma\in\C$, $(S, \circ_\gamma)$ is a simple
anti-pre-Lie algebra in the sense that there is not an ideal
besides zero and itself since the sub-adjacent Lie algebra of any
ideal of $(S, \circ_\gamma)$ is also an ideal of the Witt algebra
$\WW$ and $\WW$ is a simple Lie algebra. \item For any
$\gamma\in\C$, it is straightforward to show that
$(S,\circ_\gamma)$ is an admissible Novikov algebra. Consequently,
any graded admissible Novikov algebraic structure on
$\WW$ is defined by Eq.~\eqref{main-1}. 
\item Recall that Novikov algebras are pre-Lie algebras whose right multiplication operators are commutative.
Now suppose that
$(S,\cdot)$ is a graded Novikov algebraic structure on
$\WW$. Then
$$W_n\cdot W_m-W_m\cdot W_n=[W_n, W_m]=(m-n)W_{m+n}, \ \ \ \forall m,n\in\Z.$$
Note that the $2$-algebra of $(S,\cdot)$ is an
admissible Novikov algebra, denoted by $(S,\circ)$. 
So there exists $\gamma\in\C$ such that
$$W_n\circ W_m=(\gamma+m+2n)W_{m+n}, \ \ \ \forall m,n\in\Z.$$
Hence
$$W_n\cdot W_m=-\frac{1}{3}(W_n\circ W_m-2W_m\circ W_n)=(\frac{1}{3}\gamma+m)W_{m+n}, \ \ \forall m,n\in\Z.$$
So any graded Novikov algebraic structure on $\WW$ satisfies
$$W_n\cdot W_m=(\xi+m)W_{m+n}, \ \ \forall m,n\in\Z,$$
where $\xi\in\C$.
This agrees with \cite[Example 3.2]{KCB}.
\end{enumerate}
\end{remark}

At the end of this section, we give the classification of the
 graded anti-pre-Lie algebras on the Witt algebra $\WW$
obtained in Theorem~\ref{main} in the sense of isomorphisms. At
first, recall the following result.
\begin{lemma}\label{Witt-aut}\textup{(\cite{Bav})}
Let $T$ be an automorphism of the  Witt algebra $\WW$. Then there
exist $\lambda\in\C^*, \epsilon\in\{\pm 1\}$ such that
$$T(W_m)=\epsilon \lambda^mW_{\epsilon m}, \ \ \ \ \forall m\in\Z.$$
\end{lemma}

\begin{proposition}\label{iso}
Let $(S_1, \circ)$ and $(S_2, \bullet)$ be two graded
anti-pre-Lie algebraic structures on $\WW$ defined by
$$W_n\circ W_m=\varphi_1(n,m)W_{m+n}, \ \ \forall m,n\in\Z, \ \
 \text{\ and \ } W_n\bullet W_m=\varphi_2(n,m)W_{m+n}, \ \ \ \forall m,n\in\Z,$$ respectively,
 where $\varphi_1:\Z\times \Z\rightarrow \C,\varphi_2:\Z\times \Z\rightarrow \C$ are complex-valued functions.
Then $(S_1, \circ)$ and $(S_2, \bullet)$ are isomorphic as anti-pre-Lie algebras if and only if
$$\varphi_1(n, m)=\varphi_2(n, m),\ \ \forall m,n\in\Z, \ \  \text{\ or \ } \varphi_1(n, m)=-\varphi_2(-n, -m),\ \ \forall m,n\in\Z.$$
\end{proposition}
\begin{proof}
$(\Longleftarrow)$. If $\varphi_1(n, m)=\varphi_2(n, m)$ for any
$m,n\in\Z$, then it is straightforward to show that $(S_1, \circ)$
and $(S_2, \bullet)$ are isomorphic as anti-pre-Lie algebras.
Suppose that $\varphi_1(n, m)=-\varphi_2(-n, -m)$ for any
$m,n\in\Z$. Define a linear map $T: S_1\rightarrow S_2$ by $$
W_m\mapsto -W_{-m}, \ \ \ \ \ \forall m\in\Z.$$ Then
straightforwardly  we show that $T$ is an anti-pre-Lie algebra
isomorphism from $(S_1, \circ)$ and $(S_2, \bullet)$.

$(\Longrightarrow)$. Suppose $T$ is an anti-pre-Lie algebra
isomorphism from $(S_1, \circ)$ and $(S_2, \bullet)$. Then
\begin{align*}
T([W_n, W_m])&=T(W_n\circ W_m)-T(W_m\circ W_n) \\
&=T(W_n)\bullet T(W_m)-T(W_m)\bullet T(W_n)
=[T(W_n), T(W_m)], \ \ \ \forall m,n\in\Z.
\end{align*}
So $T$ is an automorphism of $\WW$. Thus by Lemma \ref{Witt-aut},
there exist $\epsilon\in{\pm 1}, \lambda\in\C^*$ such that
$$T(W_m)=\epsilon \lambda^mW_{\epsilon m}, \ \ \ \ \forall m\in\Z.$$
\setcounter{case}{0}
\begin{case}\label{epsilon=1}
$\epsilon=1$.
\end{case}
In this case, $T(W_m)=\lambda^mW_m$ for any $m\in\Z$. Thus for any
$m, n\in\Z$, we have
\begin{align*}
\lambda^{m+n}\varphi_1(n, m)W_{m+n}=\varphi_1(n, m)T(W_{m+n})&=T(W_n\circ W_m)  \\
&=T(W_n)\bullet T(W_m)=\lambda^{m+n}\varphi_2(n,m)W_{m+n}.
\end{align*}
Hence $$\varphi_1(n, m)=\varphi_2(n, m), \ \ \ \forall m,n\in\Z.$$
\begin{case}\label{epsilon=-1}
$\epsilon=-1$.
\end{case}
In this case, $T(W_m)=-\lambda^mW_{-m}$ for any $m\in\Z$. So for
any $m, n\in\Z$, we have
\begin{align*}
-\lambda^{m+n}\varphi_1(n, m)W_{-m-n}=\varphi_1(n, m)T(W_{m+n})&=T(W_n\circ W_m)  \\
&=T(W_n)\bullet T(W_m)=\lambda^{m+n}\varphi_2(-n,-m)W_{-m-n}.
\end{align*}
So we have $$\varphi_1(n, m)=-\varphi_2(-n, -m), \ \ \ \forall m,n\in\Z.$$

Therefore the conclusion follows.
\end{proof}

Combining Theorem \ref{main} and Proposition \ref{iso} together,
we have the following conclusion.
\begin{theorem}\label{main-iso} With the assumptions and the
notations in Theorem \ref{main}, let $\gamma_1,\gamma_2\in\C$, the
graded anti-pre-Lie algebraic structures
$(S,\circ_{\gamma_1})$ and $(S,\circ_{\gamma_2})$ on the Witt
algebra $\WW$ are isomorphic if and only if
$$\gamma_1=\gamma_2, \ \  \text{\ or\ \ \ } \gamma_1=-\gamma_2.$$
\end{theorem}

\section{Graded anti-pre-Lie algebraic structures on the Virasoro algebra}
We investigate the graded anti-pre-Lie algebraic
structures on the Virasoro algebra $\VV$ satisfying certain
natural conditions.

In this section, since the Virasoro algebra $\VV$ is the central
extension of the Witt algebra $\WW$, it is natural to consider the
 graded anti-pre-Lie algebraic structures on $\VV$
satisfying
\begin{align}
&W_n\circ W_m=\phi'(n,m)W_{m+n}+\phi(n,m)\delta_{m+n,0}{\bf c},  \label{Vir-1}\\
&W_n\circ {\bf c}={\bf c}\circ W_n={\bf c}\circ {\bf c}=0, \ \ \ \ \forall m,n\in\Z, \label{Vir-2}
\end{align}
 where $\phi':\Z\times \Z\rightarrow \C$ and $\phi:\Z\times
\Z\rightarrow \C$ are complex-valued functions.

\begin{lemma}\label{lemma:4.1}
If there is a graded anti-pre-Lie algebraic structure
on $\VV$ satisfying Eqs.~\eqref{Vir-1} and \eqref{Vir-2}, then
$\phi'$ satisfies Eqs.~\eqref{linear-1} and \eqref{linear-2}.
Therefore
\begin{equation}\phi'(n,m)=-(\gamma+m+2n),\;\;\forall m,n\in \Z,\end{equation} for some $\gamma\in\C$.
\end{lemma}

\begin{proof} By the assumption, we get
\begin{align*}
&(m-n)W_{m+n}+\delta_{m+n,0}\frac{n^3-n}{12}{\bf c}=[W_n, W_m]=W_n\circ W_m-W_n\circ W_m  \\
&=(\phi'(n,m)-\phi'(m,n))W_{m+n}+(\phi(n,m)-\phi(m,n))\delta_{m+n,0}{\bf c},\ \ \ \forall m,n\in\Z.
\end{align*}
Thus $\phi'$ satisfies Eqs.~\eqref{linear-1} by observing the coefficients of $W_{m+n}$.
Also, we can deduce $\phi'$ satisfies Eqs.~\eqref{linear-2} by comparing the coefficients of $W_{m+n+l}$ in the following equation
$$W_m\circ(W_n\circ W_l)-W_n\circ(W_m\circ W_l)=(W_n\circ W_m)\circ W_l-(W_m\circ W_n)\circ W_l,\ \ \forall m,n,l\in\Z.$$
 Therefore the
conclusion follows from Lemma~\ref{lem:3.1} and
Theorem~\ref{main}.
\end{proof}

In this case, we rewrite Eq.~\eqref{Vir-1} as
\begin{equation} \label{Vir-3'} W_n\circ
W_m=-(\gamma+m+2n)W_{m+n}+\phi(n,m)\delta_{m+n,0}{\bf c}, \ \ \
\forall m,n\in\Z.\end{equation}

\begin{theorem}\label{thm-Vir}
There does not exist a graded anti-pre-Lie algebraic
structure on the Virasoro algebra $\VV$ satisfying
Eqs.~\eqref{Vir-1} and \eqref{Vir-2}.
\end{theorem}
\begin{proof}
Assume that there exists a graded anti-pre-Lie
algebraic  structure on $\VV$ satisfying Eqs.~\eqref{Vir-1} and
\eqref{Vir-2}. By Lemma~\ref{lemma:4.1}, Eq.~\eqref{Vir-3'} holds.
By definitions of anti-pre-Lie algebras and the Virasoro algebra,
$\phi:\Z\times \Z\rightarrow \C$ satisfies
\begin{align}
&\phi(m,-m)-\phi(-m,m)=\frac{m^3-m}{12},\label{Vir-3}\\
&(\gamma+m-n)\phi(n,-n)-(\gamma+n-m)\phi(m,-m)=(m-n)\phi(m+n,-m-n), \label{Vir-4} \ \ \forall m,n\in\Z.
\end{align}
Taking $n=0$ in Eq.~\eqref{Vir-4}, we get
\begin{align}\label{Vir-5}
(\gamma+m)\phi(0,0)=\gamma \phi(m,-m),\ \ \ \forall m\in\Z.
\end{align}
\setcounter{case}{0}
\begin{case}
$\gamma=0$.
\end{case}
In this case, $\phi(0,0)=0$. Taking $m=-n\ne 0$ in
Eq.~\eqref{Vir-4}, we have
$$2m\phi(-m,m)+2m\phi(m,-m)=2m\phi(0,0)=0, \  \ \forall m\in\Z^*.$$
Hence
\begin{align}\label{Vir-6}
\phi(m,-m)+\phi(-m,m)=0, \ \ \ \forall m\in\Z^*.
\end{align}
By Eqs.~\eqref{Vir-3}, \eqref{Vir-6} and the fact that
$\phi(0,0)=0$, we have
$$\phi(m,-m)=\frac{m^3-m}{24},\ \  \forall m\in\Z.$$
Therefore, by Eq.~\eqref{Vir-4}, we have
$$(m-n)\frac{m^3+n^3-m-n}{24}=(m-n)\frac{(m+n)^3-m-n}{24}, \ \ \ \forall m,n\in\Z,$$
which is a contradiction.
\begin{case}
$\gamma\ne 0$.
\end{case}
In this case, by Eq.~\eqref{Vir-5}, we have
$$\phi(m,-m)=\frac{\gamma+m}{\gamma}\phi(0,0),$$ for any $m\in\Z$. So
$$\frac{\gamma+m}{\gamma}\phi(0,0)-\frac{\gamma-m}{\gamma}\phi(0,0)=\phi(m,-m)-\phi(-m,m)=\frac{m^3-m}{12}, \ \forall m\in\Z.$$
Therefore we have $$\phi(0,0)=\frac{m^2-1}{24}\gamma,\ \ \ \forall
m\in\Z^*,$$  which is a contradiction.

Hence the conclusion holds.
\end{proof}

\noindent {\bf Acknowledgements.} This work is supported by NSFC
(11931009, 12271265, 12261131498, 12326319, 12401032, W2412041), China National Postdoctoral
Program for Innovative Talents (BX20220158), China
Postdoctoral Science Foundation (2022M711708), Fundamental
Research Funds for the Central Universities and Nankai Zhide
Foundation. The authors thank the referees for their helpful comments and suggestions.

\noindent {\bf Declaration of interests. } The authors have no
conflicts of interest to disclose.

\noindent {\bf Data availability. } Data sharing is not applicable
as no new data were created or analyzed.

\end{document}